\documentclass[11pt,a4paper]{article}

\usepackage{amsthm}
\usepackage{amssymb}
\usepackage{amsmath}
\newtheorem{thm}{Theorem}
\newtheorem{defn}{Definition}
\newtheorem{lem}{Lemma}
\newtheorem{cor}{Corollary}

\newtheorem{xremark}{Remarks}
\newtheorem{prop}[thm]{Proposition}

\newtheorem{step}{Step}
\newcommand{\thmref}[1]{Theorem~\ref{#1}}

\newcommand{\lemref}[1]{Lemma~\ref{#1}}
\newcommand{\propref}[1]{Propotion~\ref{#1}}
\newcommand{\defnref}[1]{Definition~\ref{#1}}

\begin{document}
\title{Integration of algebroidal functions \footnote
{The research is  supported by the National Natural Science
Foundation of China(No.11501127), Guangdong Natural Science
Foundation(No.2015A030313628), Training plan for the  Distinguished
Young Teachers in Higher Education of Guangdong(No.Yqgdufe1405) and Foundation for Distinguished Young Talents in Higher Education of Guangdong Province (No.2014KQNCX068). Corresponding author:Kong Yinying}}

\author{
Sun Daochun $^1$, Huo Yingying$^2$, Kong Yinying$^{3}$, Chai Fujie $^4$\\
{\small 1.School of Mathematics,  South China Normal University,
Guangzhou, 510631, China}\\
{\small 2.School of Applied Mathematics, Guangdong University of
Technology, Guangzhou, 510520, China}\\
{\small 3.School of  Mathematics and Statistics, Guangdong University of Finance and Economics, Guangzhou,510320 China}\\
{\small 4.Guangzhou Civil Aviation College, Guangzhou, 510403, China}\\
{\small 1.~\ 1457330943@qq.com},
{\small 2.~\ fs$\_$hyy@aliyun.com},\\
{\small 3.~\ kongcoco@hotmail.com},
{\small 4.~\ chaifujie@sina.com}}

\maketitle

\begin{abstract}
  In this paper, we introduce the integration of algebroidal functions on Riemann surfaces for the first time. Some properties of integration are obtained.
By giving the definition of residues and integral function element, we obtain the condition that the integral is independent of path. At last, we prove that the integral of an irreducible algebroidal function is also an irreducible algebroidal function if all the residues at critical points are zeros.
\end{abstract}

\noindent {\bf AMS} Mathematics Subject Classification(2000): 30D30,
30D35

\section{Introduction}
Algebroidal functions are a kind of important multi-value functions. For example, in the field of complex differential equation, algebroidal solutions are more general than meromorphic solutions. But there are few results on algebroidal functions for lack of effective tools. Though integration is a basic definition, we have not seen any research on the integral of algebroidal functions, even in \cite{He,David}. Note that  the integral of algebroid functions  we define is completely different from  \cite{Davenport}. For holomorphic functions, the integrals can be defined as the limit of the same kind as that encountered in a usual integral. But for algebroidal functions, because of its multi-valuedness, we can not define the integral as what we usually do for holomorphic  functions. In this paper, we see algebroidal functions as single-valued on the Riemann surfaces. The definite integral on the Riemann surface is defined and some of its properties are obtained in \cite{farkas,lv}. Then by giving the definition of residues at critical points, we prove that the integral is independent of path if all the residues at critical points are zeros. Under the same assumption, we define the integral function element and prove that they are unchangeable for direct continuation. And last, we obtain the integral of an irreducible algebroidal function is also an irreducible algebroidal function if all the residues at critical
points are zeros.
\par

First, we provide the definition of algebroidal functions.
Let $A_1(z)$, $A_2(z)$, ..., $A_k(z)$ be a group of meromorphic functions in the complex plane $\mathbf{C}$, then the equation
in two variables
\begin{equation}\label{eq1}\Psi(W,z)= W^k+A_1(z)W^{k-1}+...+A_k(z)=0
\end{equation}
defines a $k$-valued \emph{algebroidal function} $W(z)$ in $\mathbf{C}$.

The equation \eqref{eq1} is irreducible if it can not be expressed as the product of two non-meormorphic functions. In this case, we say the
algebroidal function is irreducible. In this paper, we confine our consideration on irreducible algebroidal functions. We use the standard
definitions and notations of algebroidal functions ; e.g., see \cite{He,sun1,sun2,yangl,sun3}.\par

The resultant of $\Psi(W,z)$ and its partial derivative $\Psi_W(W,z)$, which
 is said to be the discrimination of $W(z)$, is denoted by $R(\Psi,\Psi_W)(z)$.
For an irreducible algebroidal function, we have  $R(\Psi,\Psi_W)(z)\not\equiv 0$. Hence, points in the complex plane can be divided into two kinds, say
critical points and regular points. By \emph{critical points}, we mean points in the set $S_W=:\{z;R(\Psi,\Psi_W)(z)\neq 0\}\cup\{z;z~ \text{is~the~pole~of~some} ~A_j(z), ~j=0,1,\cdots,k\}$. And points in the set $T_W=:\mathbf{C}-S_W$ are \emph{regular points}. It can be deduced that critical points are isolated. \par

\begin{defn}\label{d2}
By a \emph{function element} $(p(z),D_a)$, or $(p(z),a)$, we mean a simply-connected domain $D_a$ including $a\in\mathbf{C}$ and a holomorphic function $p(z)$ in $D_a$.
Two function elements $(p(z),a)$ and $(q(z),b)$ are \emph{equal}, if $a=b$ and there is a neighborhood $U$ of $a$ such that $p(z)=q(z)$ in $U$.
If for all $z\in D_a$, we have $\Psi(z,p(z))=0$, then $(p(z),D_a)$ is said to be a function element of the algebroidal function $W(z)$.\end{defn}
Suppose $W(z)$ is an irreducible algebroidal function defined by \eqref{eq1}. If there is an ordered pair $(w_0,z_0)$ satisfying
\begin{enumerate}
\renewcommand{\labelenumi}{(\roman{enumi})}
\item
$\Psi(w_0,z_0)=0$,

\item $\Psi_W(w_0,z_0)\neq0$,
\end{enumerate}
by the implicit function theorem,
then there uniquely exists a function element $(w(z),D_{z_0})$ such that $w(z_0)=w_0$ and $\Psi( w(z),z)\equiv0$ for all $z\in D_{z_0}$ (see \cite{lv}).
Or, the ordered pair $(w(z), D_{z_0})$ is a function element of the algebroidal function $W(z)$. We denote the set of all the function element of the algebroidal function $W(z)$ by $\widetilde{T}_W$.

\begin{defn}\label{d3}
A function element $(q(z),b)$ is said to be the \emph{direct continuation} of $(p(z),a)=(p(z),D_a)$, if two conditions hold:
\begin{enumerate}
\renewcommand{\labelenumi}{(\roman{enumi})}
\item
$b\in D_a$,

\item there is a neighborhood $U$ of $b$, such that $U\in D_a$ and $p(z)=q(z)$ for all $z\in U$.
\end{enumerate}
Hence, it can be denoted by $(p(z),b)=(q(z),b)$.\end{defn}
By this definition, we have for all $u\in D_a$, function element $(p(z),u)$ is a direct continuation of $(p(z),a)$.
\begin{xremark}
\begin{enumerate}
\renewcommand{\labelenumi}{(\roman{enumi})}
\item
~If $(w(z),a)$ is a function element of the algebroidal function $W(z)$, by the uniqueness of direct continuation, all its direct continuations are
belong to $W(z)$.

\item Function element $(w(z),D_a)$ and all its direct continuations $\{(q(z),b)\}$ $(b\in D_a)$ form a neighborhood $U_w$ of $(w(z),D_a)$ on the Riemann
surface. And the function $w[(q(z),b)]=:w(b)\in \mathbf{C}$ is an analytic function in $U_w$.
\end{enumerate}
\end{xremark}

The critical points have been so far excluded from our considerations. Next we will consider the points in $S_W$.
For any point $a\in S_W$, there is a element $(q(z), U_a)=(q(z),a)$, where $q(z)$ can be written as an Pusieux series
\begin{equation}\label{eqx1}
 q(z)=\sum^\infty_{n=u}B_{n}(z-a)^{n/m},~~~B_u\ne 0.\end{equation}
 Especially, when $m=1, u\geq0$, the element $(q(z),U_a)$ is the function element we mentioned above. When $u<0$, the element is said to be
 a \emph{pole element}. When $m>1$, we say that it is an \emph{algebraic element} and $a$ is a \emph{branch point} of order $m-1$. Both pole elements and algebraic elements are said to be \emph{singular elements}.

 ~\\
 From \cite{farkas}, we obtain that the domain of an irreducible algebroidal function is a Riemann surface composed by function elements and singular elements.
We then define the integrals on these Riemann surfaces.

\section{Definite integrals of algebroidal functions}

\begin{defn}\label{d4}~
Suppose that $\widetilde{\mathbf{C}}$ is the Riemann surface defined by a $k$-valued algebroidal function $W(z)$ and $\widetilde{L}$ is an arc on $\widetilde{\mathbf{C}}$ satisfying
$$\widetilde{L}:~~(w_t(z),z(t))~~(\alpha\leq
t\leq\beta)~~\subset ~~\widetilde{T}_W.
$$
The elements $(w_0(z),a=z(\alpha))$ and $(w(z),b=z(\beta))$ are the initial point and terminal point of $\widetilde{L}$, respectively. We choose $n$ points
along $\widetilde{L}$  arbitrarily:
$$(w_0(z),  z({t_0})=z(\alpha)),  ~(w_1(z),  z({t_1})),  ~ (w_2(z),  z({t_2})),  ~ \cdots,$$
  $$ (w_{n-1}(z),  z({t_{n-1}})),  ~(w(z),  z({t_n})=z(\beta)). $$
Since $W(z)$ is continuous on $\widetilde{L}$, when $\lambda=\max\limits_{0\leq j\leq n-1}|t_{j+1}-t_{j}|$ tends to zero,
 the limit
 $$J=:\lim\limits_{\lambda\rightarrow\infty}S_{n}=\sum\limits_{j=0}^{n-1}w_{j}(z(t_j))(z({t_{j+1})}-z({t_j}))
$$
exists. And we call it the \emph{definite integral} of $W(z)$ along $\widetilde{L}$.

\end{defn}

It can be deduced from \defnref{d4} that the followings hold.
\begin{prop}\label{p1}
\begin{enumerate}
\renewcommand{\labelenumi}{(\roman{enumi})}
\item$\int_{\widetilde{L}}\alpha W(z)dz=\alpha\int_{\widetilde{L}}W(z)dz$, where $\alpha\in\mathbf{C}$;

\item $$\int_{\widetilde{L}}W(z)dz=\int_{\widetilde{L}_{1}}W(z)dz+\int_{\widetilde{L}_{2}}W(z)dz
    +\cdots+\int_{\widetilde{L}_{n}}W(z)dz,$$
where $\widetilde{L}$ can be subdividing into $n$ subarcs $\widetilde{L}_i~(i=1,2,\cdots,n)$;

\item $$\int_{\widetilde{L}^{-}}W(z)dz
=-\int_{\widetilde{L}}W(z)dz,$$ where $\widetilde{L}^{-}$ is the opposite path of $\widetilde{L}$;

\item If $\widetilde{L}$ is a close path on the Riemann surface and there is no singular element in it, then
    $$\int_{\widetilde{L}}W(z)dz=0.$$

\end{enumerate}
\end{prop}

\subsection{Independence of path}
\begin{defn}\label{d5}
Suppose $(q(z),a)$ is a singular element of an algebroidal function $W(z)$,
where
$$q(z)=\sum^\infty_{n=u}B_{n}(z-a)^{n/m},~~~B_u\ne 0, m\geq 1 ~(\text{or}~u<0 ).$$
The complex number $m\cdot B_{-m}$ is called the \emph{residue} of the singular element $(q(z),a)$.
\end{defn}

\begin{lem}\label{l2} Suppose that $\widetilde{L}$ is a closed path on the Riemann surface with
only one singular element $(q(z),a)=(q(z),D_a)$ inside of it. If the residue of the singular element $(q(z),a)$ is
$mB_{-m}=0$, then
$$\int_{\widetilde{L}} W(z)dz=0.$$

\end{lem}
\begin{proof} Let $\varepsilon\in (0,r)$ be sufficiently small and $\widetilde{\Gamma}\subset\widetilde{T_W}$ be a closed curve on the Riemann surface
satisfying
$$\widetilde{\Gamma}_m:(w_{z(t)}(z), ~z(t)=a+\varepsilon e^{it})~(0\leq t\leq2m\pi).$$
By \propref{p1} and \defnref{d4}, we have
\begin{eqnarray*}&& \int_{\widetilde{L}}
W(s)ds=\int^{2m\pi}_{0}
\sum^\infty_{n=u}B_{n}r'^{n/m} e^{itn/m}~\cdot~  ir'
e^{it}dt\\
&&=i\sum^\infty_{n=u}B_{n} r'^{1+\frac{n}{m}}\int^{2m\pi}_{0}
e^{it(1+\frac{n}{m})}dt\\
&&=i B_{-m}\int^{2m\pi}_{0} dt=i B_{-m}2m\pi=0.
\end{eqnarray*}
\end{proof}

\begin{thm}\label{t4}(Independence of path)~Suppose that $\widetilde{\mathbf{C}}$ is a Riemann surface defined by
an irreducible algebroidal function $W(z)$ and the residue of every singular element on it is zero. Then for any close path $\widetilde{L}\subset \widetilde{T}_W$, we have
$$\int_{\widetilde{L}} W(z)dz=0.$$
\end{thm}

\begin{proof}~Since $\widetilde{L}$ is a closed path, there are only finitely many singular elements in it, say $(q_j(z),B(a_j,r_j))~(j=1,2,\dots,n)$. By adding finitely many curves $\widetilde{\gamma}_j\subset \widetilde{T}_W~(j=1,2,\cdots,2(n-1))$ in $\widetilde{L}$,  then $\widetilde{L}$ and $\widetilde{\gamma}_j~(j=1,2,\cdots,2(n-1))$ form finitely many closed curves $\widetilde{L}_j~(j=1,2,\cdots,n)$ and there is at most one singular element inside of each $\widetilde{L}_j$. Hence by \propref{p1}, \lemref{l2} and taking account of the cancellations along the curves $\widetilde{\gamma}_j~(j=1,2,\cdots,2(n-1))$, we have
\begin{eqnarray*}
\int_{\widetilde{L}} W(z)dz
=\sum\limits_{j=1}^n\int_{\widetilde{L}_j} W(z)dz=0.
\end{eqnarray*}

\end{proof}

\begin{xremark}
An algebroidal function satisfying the hypothesises in \thmref{t4} does exist. For example, $\sqrt{z}$ is an irreducible $2$-valued algebroidal function in the complex plane. Element $(\sqrt{z},0)$ is the only singular element on its Riemann surface, whose residue is zero.
\end{xremark}

By \thmref{t4}, we have
\begin{cor}\label{c1} Suppose that $W(z)$ is an algebroidal function  and all the residues of its singular elements are both zero.  Then the define integration
$$\int_{\widetilde{L}} W(z)dz$$
depends on the initial point $(w_0(z),a)$ and the terminal point $(w(z),b)$ only. Hence the define integration
$$\int^{(w(z),b)}_{(w_0(z),a)}W(z)dz$$
is well-defined.

\end{cor}

\subsection{Integral elements}
In this part, we will define the integrals of all elements on the Riemenn surface.
First, we study the \emph{integral function elements}:
\begin{defn}\label{d6}
Suppose that $W(z)$ is an algebroidal function  and all the residues of its singular elements are both zero. Fix a function element $(w_a(z),a)$ on its corresponding Riemann surface. For any function element $(w_b(z),U_b)$, or $(w_b(z),b)$,
we define its integral function element as
\begin{eqnarray}
\label{eq2} (\int_{a}w_b(z),b):=(c_{a,b}+\int^z_{b}w_b(s)ds,~b),
\end{eqnarray}
where $\int^z_{b}w_b(s)ds$ is the complex integration and $c_{a,b}$ is the define integral
$$c_{a,b}=\int^{(w_b(z),b)}_{(w_a(z),a)}W(z)dz.$$
\end{defn}

It can be deduced from \defnref{d6} that the integral function elements are also function elements.
And we will see it is unchangeable for analytic continuation:

\begin{lem}\label{l3}
Suppose that $W(z)$ is an algebroidal function and all the residues of its singular elements are both zero.  Let function element $(w_b(z),  u)$ be a direct
continuation of function element $(w_b(z),  b)$.
Then the integral function element $(\int_{a}w_b(z),  u)$ is the direct continuation of $(\int_{a}w_b(z),  b)$.
\end{lem}

\begin{proof}
It follows from \defnref{d6} that the integral function element of $(w_b(z), b)$ is
\begin{equation}\label{eq3}
(\int_{a}w_b(z), b)\stackrel{\eqref{eq2}}{=}(c_{a,  b}+\int^z_{b}w_b(s)ds,  ~b).
\end{equation}
And the integral function element of $(w_b(z), u)$ is
\begin{eqnarray*}
(\int_{a}w_b(z), u)\stackrel{\eqref{eq2}}{=}&&(\int^{(w_b(z), u)}_{(w_0(z), a)}W(z)dz+\int^z_{u}w_b(s)ds,  ~u)\\
=&&(\int^{(w_b(z), b)}_{(w_0(z), a)}W(z)dz+\int^{(w_b(z), u)}_{(w_b(z), b)}w_b(z)dz+\int^z_{u}w_b(s)ds, u)\\
=&&(c_{a, b}+\int^u_b w_b(z)dz+\int^z_{u}w_b(s)ds, u)
\end{eqnarray*}
That is
\begin{equation}\label{eq4}
(\int_{a}w_b(z), u)=(c_{a, b}+\int^z_b w_b(s)ds, u)\\
\end{equation}
By equality \eqref{eq3}, \eqref{eq4} and  \defnref{d3} ,  we can obtain that integral function element $(\int_{a}w_b(z), u)$ is the direct continuation
of integral function element $(\int_{a}w_b(z), b)$.
\end{proof}
Hence, for all $u\in U_b$,
integral function element $(\int_{a}w_b(z), u)$ is the direct continuation of $(c_{a, b}+\int^z_{b}w_b(s)ds, b)$.
 The following lemma shows that the singular elements are \emph{weakly bounded}. This implies that
every critical point is either analytic or a pole. For completeness, we give the proof of
it. We should mention that the proof has been given in \cite{lv,sun1}.

\begin{lem}\label{l1}
 Let $W(z)$ be an algebroidal function defined by \eqref{eq1}. Then for any $z_0\in\mathbb{C}$, there exist $r>0,~M>0~\text{and}~n\in\mathbf{N}_+$, such that for function element $(p(z),a)$ satisfying $\{0<|z-a|<r\}\subset T_W$, we have
$$ |(a-z_0)^np(a)|<M.$$

 \end{lem}
\begin{proof} Since the coefficients in equation \eqref{eq1} are meromorphic functions, they can be represented as Laurent series in $\{0<|z-z_0|<r\}$, respectively.
That is
\begin{eqnarray*}
A_j(z)=\sum\limits_{l=n_j}^{+\infty} a_l^{(j)}(z-z_0)^{l},~~~a_{n_j}^{(j)}\neq 0, j=1,2,\cdots,k
\end{eqnarray*}
Let $n_0=\min\{n_j;j=1,2,\cdots,k\}$. Then there is a positive number $M$ such that for any $a\in\{0<|z-z_0|<r\}$
$$|(a-z_0)^{n_0}|(|A_1(a)|+|A_2(a)|+\cdots+|A_k(a)|)\leq M.$$
When $|p(a)|\geq 1$, we have
\begin{eqnarray*}
&&|(a-z_0)^{n_0}p(a)|=\left|(a-z_0)^{n_0}\left(A_1(a)+\frac{A_2(a)}{p(a)}+\cdots+\frac{A_k(a)}{p^{k-1}(a)}\right)\right|\\
&&\leq |(a-z_0)^{n_0}|\left(|A_1(a)|+\frac{|A_2(a)|}{|p(a)|}+\cdots+\frac{|A_k(a)|}{|p^{k-1}(a)|}\right)\\
&&\leq |(a-z_0)^{n_0}|(|A_1(a)|+|A_2(a)|+\cdots+|A_k(a)|)\leq M.
\end{eqnarray*}
When $|p(a)|<1$, we have
$$|(a-z_0)^{n_0}p(a)|<r<M.$$

\end{proof}
\begin{lem}\label{l4}(weakly bounded) ~Suppose that $W(z)$ is an algebroidal function  and all the residues of its singular elements are both zeros. Fix a function element $(w_a(z),a)$ on its corresponding Riemann surface. Then for any $z_0\in \mathbf{C}$, there exist two real positive numbers $M,~r$ and an integer $n\geq 0$ such that for any integral function element
$(\int_a w_a(z),u)$ satisfying $u\in \{0<|z-z_0|<r\}$, we have
$$ |(u-z_0)^n   \int_a w_a(u)|<M.$$

\end{lem}

\begin{proof}
Let $z_0\in S_W$. Then by \lemref{l1}, there exist $r>0$, $M>0$ and an integer $n\geq0$ such that for any function element
$(w(z),u)$ satisfying $u\in \{0<|z-z_0|<r\}$, we have
$$ |(u-z_0)^n   w(u)|<M.$$
Cutting $B_0$ into simply connected domain $B_-:=\{0<|z-z_0|<r\cap
|\arg(z-z_0)|<\pi\}\subset T_W$, we have $k$ single branches $w_1(z)$,$w_2(z)$,...,$w_k(z)$ of $W(z)$.\par
For $b\in
B_-$, there are $k$ distinct function elements $(w_1,b)$,$(w_2,b)$,...,$(w_k,b)$. By \defnref{d6}, the corresponding integral function elements are
$$
(\int_{a}w_j(z),b):=(\int^{(w_j(z),b)}_{(w_0(z),a)}W(z)dz+\int^z_{b}w_j(s)ds,~~~b),~~~j=1,2,...,k.$$
Let $$
H:=\max\{|\int^{(w_j(z),b)}_{(w_0(z),a)}W(z)dz|,~~j=1,2,...,k\}.$$
For for any $(u\in B_0)$, the integral function elements of $(w_j(z),  u))$ $(j=1,2,...,k)$ are $(\int_{a}w_j(z),  u)$ $(j=1,2,...,k)$, respectively.  It follows from $\lemref{l3}$ that they are the direct continuations of $(\int_{a}w_j(z),b)$. Hence, for any $u\in
B_-\cap \{|u|>|b|\}$, we have
$$|\int^{(w_j(z),b)}_{(w_0(z),a)}W(z)dz|+|\int^u_{b}w_j(s)ds|\leq
H+|\int^u_{b}w_j(s)ds|\stackrel{\eqref{eq3}}{\leq}
 H+ 2\pi |b|\cdot
\frac{M}{|u-z_0|^n}.
$$
\end{proof}

\section{Integral of algebroidal function}
Then we can define the integral of algebroidal function:
\begin{defn}\label{d7} ~We say $M(z)$ is the \emph{integral} of an algebroidal function $W(z)$ if $M'(z)=W(z)$.
\end{defn}

\begin{thm}\label{t1}~~Suppose that $W(z)$ is an irreducible algebroidal function and all the residues of its singular elements are both zeros. Then the integral of $W(z)$ exists and it is also irreducible.
Further, if $(w_a(z),a)$ is given, it is unique.
\end{thm}

\begin{proof}
\begin{step} For all $b\in T_W$,  there are $k$ distinct function elements $\{(w_j(z), b)=(w_j(z),  U_b)\}^k_{j=1}$.

 It follows from \defnref{d6}, there are $k$ function elements
\begin{equation}\label{eq5}(\int_{a}w_j(z),U_b):=(\int^{(w_j(z),b)}_{(w_0(z),a)}W(z)dz+\int^z_{b}w_j(s)ds,~~~ b), ~~j=1, 2, \cdots, k.
\end{equation}
By \lemref{l3}, for $u\in U_b$,
there are $k$ integral function elements
$$(\int_{a}w_j(z),u):=(\int^{(w_j(z),b)}_{(w_0(z),a)}W(z)dz+\int^z_{b}w_j(s)ds,~~~ u),  ~~j=1, 2, \cdots, k;~~u\in U_b.
$$
They are the direct continuations of $(\int_{a}w_j(z),  b)$, respectively.
Then we can define a group of functions in $U_b$£º
\begin{eqnarray*}
B_1(z):&=&-\int_{a}w_1(z)-\int_{a}w_2(z)-\cdots-\int_{a}w_k(z);\\
B_2(z):&=&\sum_{1\leq i<j\leq k}\int_{a}w_i(z)\int_{a}w_j(z);\\
\cdots&&\cdots\\
B_k(z):&=&(-1)^k\int_{a}w_1(z)\int_{a}w_2(z)\cdots\int_{a}w_k(z)
\end{eqnarray*}
It is trivial that $\{B_j(z)\}^k_{j=1}$ are single-valued analytic functions in $U_b$.
Further, they are analytic in $T_W$. \par
For any $d\in S_W$,  it is isolated.
By \lemref{l4},  $d$ is a pole  of  all $\{B_j(z)\}^k_{j=1}$ at most. Hence, we obtain $k$ meromorphic functions $\{B_j(z)\}^k_{j=1}$ in $\mathbf{C}$. \par
We form now an equation by $\{B_j(z)\}^k_{j=1}$ as follow
\begin{equation}\label{eq4}  M^k+B_1(z)M^{k-1}+...+B_k(z)=0.\end{equation}
Then the  $k$-valued algebroidal function $M(z)$ is the integral of algebroidal function $W(z)$.
\end{step}
\begin{step} We then turn to prove that $M(z)$ is irreducible. Otherwise, equation \eqref{eq4} can be composed by a $q$-valued ($q<k$)
algebroidal function $M_q(z)$. We suppose the derivative of $M_q(z)$ is defined by
$$ (W-w_1(z))(W-w_2(z))...(W-w_q(z))=W^q+d_1(z)W^{q-1}+...+d_q(z)=0.$$
Then the above equation is a factor of the equation \eqref{eq1}, which contradicts that $W(z)$ is irreducible.
\end{step}
\end{proof}

\begin{cor}\label{c2} Suppose $f(z)$ is a meromorphic function in the simply-connected domain $D$ and all the residues at poles are zeros. Then the integral of
meromorphic function  $f(z)$ exists and it is also a meromorphic function. If the initial point is given, the integral is unique.\end{cor}

\begin{thm}\label{t4}
~~Suppose that $W(z)$ is an irreducible algebroidal function and all the residues of its singular elements are both zeros. Then we can obtain a family of integral
of $W(z)$ with an arbitrarily constant.
\end{thm}

\begin{proof} For a fixed function element $(w_0(z),  a)$,  the unique equation which define the integral of algebroidal function $W(z)$ is
$$
M^k+B_1(z)M^{k-1}+\cdots+B_k(z)=(M-\int_{a}w_1(z))(M-\int_{a}w_2(z))\cdots(M-\int_{a}w_k(z))=0. $$
Then the family of integral function with an arbitrarily constant $c$ is
\begin{eqnarray*}
&&M^k+B^c_1(z)M^{k-1}+\cdots+B^c_k(z)\\
&=&[M-(c+\int_{a}w_1(z))][M-(c+\int_{a}w_2(z))]\cdots[M-(c+\int_{a}w_k(z))]\\
&=&[(M-c)-\int_{a}w_1(z)][(M-c)-M-(c+\int_{a}w_2(z)]\cdots[(M-c)-\int_{a}w_k(z)]\\
&=&(M-c)^k+B_1(z)(M-c)^{k-1}+\cdots+B_k(z)=0.
\end{eqnarray*}
Hence, we can obtain the coefficients are
\begin{eqnarray*}
B^c_1&=&-cC_k^1+B_1;\\
B^c_2&=&c^2C_k^2-cC_{k-1}^1B_1+B_2;\\
\cdots&&\cdots\\
B^c_j&=&(-1)^j[c^jC_k^j-c^{j-1}C_{k-1}^{j-1}B_1+\cdots+(-1)^jB_j];\\
\cdots&&\cdots\\
B^c_k&=&(-1)^k[c^kC_k^k-c^{k-1}C_{k-1}^{k-1}B_1+\cdots+(-1)^kB_k]\\
\end{eqnarray*}
\begin{eqnarray*}
B^c_1&=&B_1-cC_k^1;\\
B^c_2&=&B_2-cC_{k-1}^1B_1+c^2C_k^2 ;\\
\cdots&&\cdots\\
B^c_j&=&B_j-cC_{k-j+1}^1B_{j-1}+\cdots+(-c)^jC_k^j;\\
\cdots&&\cdots\\
B^c_k&=&B_k-cC_{k-k+1}^1B_{k-1}+\cdots+(-c)^kC_k^k\\
\end{eqnarray*}
\end{proof}
\makeatletter

\def\@makecol{\ifvoid\footins \setbox\@outputbox\box\@cclv
   \else\setbox\@outputbox
     \vbox{\boxmaxdepth \maxdepth
     \unvbox\@cclv\vskip\skip\footins\footnoterule\unvbox\footins}\fi
  \xdef\@freelist{\@freelist\@midlist}\gdef\@midlist{}\@combinefloats
  \setbox\@outputbox\hbox{\vrule width\marginrulewidth
        \vbox to\@colht{\boxmaxdepth\maxdepth
         \@texttop\dimen128=\dp\@outputbox\unvbox\@outputbox
         \vskip-\dimen128\@textbottom}%
        \vrule width\marginrulewidth}%
     \global\maxdepth\@maxdepth}
\newdimen\marginrulewidth
\setlength{\marginrulewidth}{.1pt}
\makeatother



\setlength{\marginrulewidth}{0pt}

\end{document}